\documentclass[leqno, 11pt, a4paper]{amsart}

\usepackage{hyperref}
\usepackage{amsmath}
\usepackage{amssymb, latexsym, slashed, amscd}
 \usepackage[all]{xy}
 \usepackage{amsthm}
\usepackage{amsfonts}
\usepackage{mathrsfs}
\usepackage{enumerate}
\usepackage{color}
  \usepackage{comment}
\usepackage{ascmac}

 \newtheorem{definition}{Definition}[section]
 \newtheorem{theorem}[definition]{Theorem}
 \newtheorem{lemma}[definition]{Lemma}
 \newtheorem{proposition}[definition]{Proposition}
 \newtheorem{corollary}[definition]{Corollary}

 \newtheorem*{theorem*}{Theorem}
\newtheorem*{proposition*}{Proposition}
\newtheorem*{lemma*}{Lemma}

 \theoremstyle{remark}
 \newtheorem{example}[definition]{Example}
 \newtheorem{remark}[definition]{Remark}

\newcommand{\op}[1]{\operatorname{#1}}

\newcommand{\Tr}{\ensuremath{\op{Tr}}}
\newcommand{\tr}{\op{tr}}

\newcommand{\Res}{\ensuremath{\op{Res}}}

\def\XXint#1#2#3{{\setbox0=\hbox{$#1{#2#3}{\int}$}
\vcenter{\hbox{$#2#3$}}\kern-.5\wd0}}

\newcommand{\Cl}{\op{Cl}}

\newcommand{\C}{\ensuremath{\mathbb{C}}} 
 
\newcommand{\N}{\ensuremath{\mathbb{N}}} 
 
\newcommand{\R}{\ensuremath{\mathbb{R}}} 
\newcommand{\T}{\ensuremath{\mathbb{T}}} 
\newcommand{\Z}{\ensuremath{\mathbb{Z}}}

\newcommand{\Ca}[1]{\ensuremath{\mathcal{#1}}}
\newcommand{\cA}{\Ca{A}}

\newcommand{\cE}{\Ca{E}}

\newcommand{\cH}{\ensuremath{\mathcal{H}}}

\newcommand{\cK}{\ensuremath{\mathcal{K}}}
\newcommand{\cL}{\ensuremath{\mathcal{L}}}

\newcommand{\cV}{\ensuremath{\mathcal{V}}}

\newcommand{\supp}{\op{supp}}

\newcommand{\Pic}{\op{Pic}}

\newcommand{\End}{\ensuremath{\op{End}}}

\newcommand{\rk}{\op{rk}}

\newcommand{\Vol}{\op{Vol}}

\newcommand{\ba}{\begin{eqnarray}}
   \newcommand{\na}{\end{eqnarray}}

\numberwithin{equation}{section}

\begin{document}

\title{Higher Nahm transform in non commutative geometry}

\dedicatory{Dedicated to Alain Connes on his $70$th  birthday}

 \author{Tsuyoshi Kato}
 \address{Department of Mathematics, Kyoto University, Japan}
 \email{tkato@math.kyoto-u.ac.jp}
 
 \author{Hirofumi Sasahira}
\address{Faculty of Mathematics, Kyushu University, Japan}
\email{hsasahira@math.kyushu-u.ac.jp}

 \author{Hang Wang}\thanks{H. Wang is supported by the Australian Research Council DE160100525}
 \address{School of Mathematical Sciences, East China Normal University, Shanghai, China}
 \address{School of Mathematical Sciences, University of Adelaide, Australia}
 \email{wanghang@math.ecnu.edu.cn\\ hang.wang01@adelaide.edu.au}


\maketitle

\section{Introduction}

An  anti self dual (ASD) connection minimizes the classical Yang-Mills 
 functional, whose minimal value is expressed in terms of the characteristic classes and
hence is a topological invariant, if an ASD connection exists.
 In non commutative geometry, notion of ASD  
does not make sense, since space does not appear  in general.
Connes reformulated the  Yang-Mills  functional 
over a compact four manifold
by using the Dixmier trace,
which reinterprets the  functional from a non commutative setting.
Connes-Yang-Mills action functional $CYM(A)$ is given by taking infimum value 
of the pre Connes-Yang-Mills action functional $\Tr_{\omega} (\nabla^2)$
in the equivalent class of non commutative connections $\nabla$
associated to $A$,
where $\Tr_{\omega}$ is the Dixmier trace.
It is not easy to formulate a non commutative Yang-Mills action functional 
without space,
even though the pre Connes-Yang-Mills action functional exists.
 Actually $CYM(A)$  is defined by using  the associated 
connections over the underlying four manifold, and its formulation is not 
straightforward from non commutative view point.

Let $\cA$ be a $*$-algebra, and $\mathcal{E}$ be a finitely generated projective
$\cA$ module. Our basic question is what kind of non commutative connections
over $\cE$ would possess `nice' properties.
Our idea is to use the Nahm transform between 
connections, which transforms an ASD connection to another ASD connection
 in the case of the $4$-dimensional  flat  torus.
Classical Nahm transform provides a kind of dual connection over
the Picard torus of covering group $\Gamma =\Z^4$ of the manifold $X$ (for example, $\Gamma=\pi_1(X)$ the fundemental group), 
and hence it concerns commutative property
of the group.
In this paper we generalize it in a non commutative way, by 
 constructing  a higher version of the 
 Connes-Yang-Mills action functional  by using 
 the Dixmier $\Gamma$-trace. 
As an output, we obtain  a finitely generated projective $C^*_r(\Gamma)$ module
 with a non commutative connection on it.
 
 Corresponding to the output in section $3$,
we quantize the group $C^*$-algebra using spectral triple and 
we need its smooth subalgebra to introduce  differential on the algebra. 
 On the other hand
we do not need to quantize the algebra on the input side in section $5$.
We actually need a Hilbert module
with a connection on it, and  does not require smooth subalgebra,
since four manifold as a space is assigned in the side. 
 
 Let $\cE$ be a Hilbert $C^*_r(\Gamma)$ module, where 
$\Gamma$ is a discrete group.
We introduce a Dixmier $\Gamma$-trace:
$$\Tr_{\omega}^{\Gamma} : \mathcal{L}^{(1, \infty)}(\cE) \to \mathbb C$$
and use it to formulate a  $\Gamma$-Connes-Yang-Mills action functional 
$CYM^{\Gamma}$.
Our motivation to construct it arose from a question how to 
formulate a higher  Nahm transform.
Let us explain it below in detail.

Let $(X,g)$ be a Riemannian 4-manifold, 
and $E \to X$ be a unitary bundle with a unitary connection $A$.
Consider $\tilde X$, a $\Gamma$-principal bundle of $X$ ($\tilde X/\Gamma=X$). A typical example is when $\tilde X$ is the universal cover of $X$ with fundamental group $\Gamma.$
Let $C^{\infty}(\Gamma)$ be a dense holomorphic closed subalgebra of $C^*_r(\Gamma)$ containing $\C\Gamma$. Then
consider the set of smooth sections
$\cE_0 = C^{\infty}(E \otimes \tilde{X} \times_{\Gamma} C^{\infty}(\Gamma))$
which is a right $C^{\infty}(\Gamma)$ module. 
 The connection
 $A$ naturally extends to a connection $\mathbb A
 $ on $\cE_0$.
In our case, the Hodge star $*$ also acts on $\Omega^2(\cE_0)$,
and hence the notion of anti-self duality on a connection makes sense.
We induce an equivalence between a minimizer 
for  $\mathbb A$ of  the $\Gamma$-Connes-Yang-Mills action functional and 
the  ASD condition. 
 This follows from the next coincidence
  combined with  theorem $14$ in \cite{ConnesYM}:
 \begin{theorem}\label{abelian}
 Let $X$ be a closed, oriented, spin, smooth $4$-manifold with $b_1=4$. Let $\Gamma=\Z^4$ be a covering group of $X.$  
 Then the Connes-Yang-Mills functional and its higher analogue coincide:
 $$CYM(A) = CYM^{\Gamma}(\mathbb A).$$
 \end{theorem}

  Nahm transform is roughly described as below.
 Suppose $X$ is spin and $D_A$ is the 
 Dirac operator 
  on $S \otimes E : = S_E$, where $S$ is the spinor bundle on $X$.
 ind $D_A$,  the index
 of the twisted Dirac operator $D_A$ over 
 $S_E \otimes \tilde{X} \times_{\Gamma} C^*_r(\Gamma)$
 consists of a formal difference of 
  finitely generated projective $C^*_r(\Gamma)$ modules,
  and is called the {\em higher index} in non commutative geometry.
 
 For $\Gamma = \Z^m$, the twisted Dirac operator $D_A$
 is regarded as a  family of Dirac operators over the Picard torus
 as in the classical case.
 Suppose ker $D_A =0$ at each parameter value,
 and hence ind $D_A =  -[$ coker $D_A]$ holds.
 In the case
  coker $D_A \subset \cH^- : = L^2(S^-_E \otimes \tilde{X} \times_{\Gamma} C^*_r(\Gamma))$
 forms a vector bundle over the Picard torus.
Take the connection $\hat{d}_A$
given by the trivial connection composed 
with the orthogonal projection onto  coker $D_A$.
The assignment:
$$\coprod_{\rho \in \Pic} \
(E \otimes L_{\rho},A) \to ( \text{ coker } D_A, \hat{d}_A)$$
is called the  Nahm transform.
So it assigns the higher index with the induced connection.
The following is known:
\begin{lemma}\label{cNT}
 Let $X$ be a closed, oriented, spin, smooth $4$-manifold with $b_1=4$. Then the Nahm transform assigns an ASD connection to another ASD connection.
\end{lemma}

The higher index exists for a general discrete group, 
where we will have no reasonable parameter space,
while  $C^*$-algebra module exists.
In the absence of  the underlying space, we have to use
a non commutative connection to formulate higher Nahm transform.
With a spectral triple, 
there is a  connection on a Hilbert $ C^*_r(\Gamma)$ module,
which is induced from the trivial one.
By using the orthogonal projection onto  coker $D_A$, 
we obtain a non commutative  connection 
 $\hat{d}_A$ as the induced connection on 
 coker $D_A$.
So we have given a higher Nahm transform
from a bundle with a connection  $(E,A)$
to a finitely generated projective $C^*_r(\Gamma)$ 
module $ \text{coker } D_A$ with a non commutative connection $ \hat{d}_A$:
 $$(E,A) \to (E \otimes  \tilde{X} \times_{\Gamma} C^*_r(\Gamma)
 ,A) \to ( \text{ coker } D_A, \hat{d}_A).$$

Lemma \ref{cNT} can be reinterpreted as below in terms our formulation:
 \begin{corollary}
 Let $X$ be a closed, oriented, spin, smooth $4$-manifold with $b_1=4$. Let $\Gamma=\Z^4$ be a covering group of $X.$  

The higher Nahm transform sends the minimizer of the 
higher Connes-Yang-Mills functional  to the minimiser 
of the Connes-Yang-Mills functional.
\end{corollary}

There arise two questions:

\vspace{2mm}

{\em Question:}

\vspace{2mm}

$(1)$
Let $X$ be a compact four manifold
with $\pi_1(X) = \Gamma$, and $E \to X$ be a unitary bundle.
Consider a number:
$$\tau_{\Gamma} (X) : = \inf_{A: \text{ ASD}}
 \ \Tr_{\omega}^{\Gamma}(\nabla_{\hat{d}_A}^2)  
\ \in \  \R.$$
When is $\tau_{\Gamma} (X)$  a topological number ?

\vspace{2mm}

$(2)$ Compare it with another number:
$$\tau_{\Gamma}' (X) : = \inf_{A} 
 \ \Tr_{\omega}^{\Gamma}(\nabla_{\hat{d}_A}^2)  
\ \in \  \R.$$
A pri-ori the inequality holds:
$$\tau_{\Gamma}' (X) \leq \tau_{\Gamma} (X)$$
if an ASD connection exists.
When  the equality holds ?

\vspace{2mm}

In the case of the $4$-dimensional
flat torus, we know both the answers affirmatively.

\section{Nahm transform}
In this section we quickly review some basic things on the Nahm transform.

\subsection{Complex surface}

Let $X$ be a 4-manifild with Riemannian metric $g$.  Suppose that $X$ has a complex structure $J$ compatible with $g$ and define a $(1,1)$-form $\omega$ by
\[
        \omega(a, b) = g(a, Jb). 
\]
We write   $\Lambda^i_X $,   $\Lambda^{i, j}_X$ for  $\Lambda^i T^*X$,  $\Lambda^{i, j} (T^* X \otimes \C)$.  Then we have the decompositions
\[
    \Lambda^2_X \otimes \C =  \Lambda^{2, 0}_X  \oplus \Lambda_X^{1, 1} \oplus  \Lambda_X^{0, 2}. 
\]
and
\[
   \Lambda^{1, 1}_{X} = (\Lambda^0_X \otimes \C) \cdot \omega  \oplus \Lambda^{1, 1}_{X, 0}. 
\]
Here $\Lambda_{X, 0}^{1, 1}$ is the orthogonal complement of $\omega$ in $\Lambda^{1,1}_X$. 
We also have
\begin{equation}   \label{eq decomp Lambda}
   \begin{split}
   & \Lambda^+_{X} \otimes \C =  
       \Lambda^{2, 0}_{X} \oplus \Lambda^{0,0}_{X}  \cdot \omega \oplus \Lambda^{0, 2}_X,  \\
  & \Lambda^{-}_{X} \otimes \C = \Lambda^{1,1}_{X, 0}.
   \end{split}
\end{equation}
Here $\Lambda^{+}_X$ is the self-dual part of $\Lambda^2_X$ and $\Lambda^{-}_X$ is the anti-self dual  of $\Lambda^2_X$.

\begin{proposition}
Let $A$ be a unitary connection on a Hermitian complex vector bundle $E$ on $X$. Then $A$ is ASD if and only if $F_A$ is $(1, 1)$-type and $(F_A, \omega) = 0$ at each point.   
\end{proposition}

Note that $F_A$ is $(1, 1)$-type if and only if $\bar{\partial}_A$ defines a holomorphic structure on $E$. 
So we have the following:

\begin{proposition}
Let $A$ be a unitary connection on a Hermitian complex vector bundle $E$ and suppose that $\bar{\partial}_A$ defines a holomorphic structure on $E$.  Then $A$ is ASD if and only if $(F_A, \omega) = 0$ at each point of $X$. 
\end{proposition}

We will consider the ASD equation on  $\R^4$ with the standard Rimeannian metric. The complex structures $J$ on $\R^4$ compatible with the Riemannian structure are parametrized by $SO(4)/U(2) ( \cong S^2)$.  It is easy to see that

\begin{equation*} \label{eq Lambda^-}
         \Lambda^{-} \otimes \C =  \bigcap_{J \in S^2}  \Lambda^{1,1}_{J }.
\end{equation*}
 
Hence we have

\begin{proposition}  \label{prop (1,1) type}
Let $E$ be a  Hermitian complex vector bundle on $\R^4$.  A unitary  connection $A$ on $E$ is ASD if and only if $F_A$ is $(1, 1)$ type for all complex structures $J$ compatible with  the Riemannian metric. 
\end{proposition}

\subsection{Nahm transform}  \label{section Nahm commutative}

Let $X$ be a closed,  Riemannian 4-manifold. 
Take a spin structure $\frak{s}$ on $X$.  Then we have the spinor bundles $S^{\pm}$ on $X$.   Choose a complex vector bundle $E$ with a Hermitian metric.   For each unitary connection $A$ on $E$ we have the twisted Dirac operator
\[
      D_A : \Gamma(S^+ \otimes E) \rightarrow \Gamma(S^- \otimes E).
\]
Note that if 
\begin{equation}    \label{eq ker D^*}
   \ker (D_A : \Gamma(S^+ \otimes E) \rightarrow \Gamma(S^- \otimes E)) = 0, 
\end{equation}
then  $D_A^* $ is surjective.   

Suppose that we have a continuous family $\{ A_{y} \}_{y \in Y}$ of unitary connections on $E$.  We assume that $A_y$ satisfies (\ref{eq ker D^*}) for all $y \in Y$.  Since $D_{A_y}^*$ is surjective,  the dimension of $\ker D_{A_y}^*$ is constant and 
\[
        \hat{E} :=  \coprod_{y \in Y} \ker D_{A_{y}}^*
\]
defines a subbundle of the trivial Hilbert bundle 
\[
        \underline{H} : = Y \times L^2(S^- \otimes E)
\] 
over $Y$ with fiber $H = L^2(S^- \otimes E)$. We have the covariant derivative corresponding to the trivial connection:  
\[
    \nabla : \Gamma( \underline{H} ) \rightarrow \Gamma( \underline{H} \otimes T^*X). 
\]
Let $i : \hat{E} \hookrightarrow \underline{H}$ be the inclusion and $p : \underline{H} \rightarrow \hat{E}$ be the $L^2$-projection. Then we get a connection $\hat{A}$ on $\hat{E}$ with covariant derivative 
\[
               \nabla_{\hat{A}} = p \nabla i. 
\]
We call $\hat{A}$ the Nahm transform of $\{ A_y \}_{y \in Y}$.

\subsection{Nahm transform of ASD connections on a 4-manifold with $b_1 = 4$}

We will follow the discussion in Section 3.2 of \cite{DK}. 

Let $X$ be a closed, oriented, spin, smooth $4$-manifold with $b_1 = 4$. 
 We denote by $\hat{X}$ the Picard torus:
\[
      \hat{X} := H^1(X;\R) / H^1(X;\Z)   \cong T^4. 
\]
We can think of $\hat{X}$ as the moduli space of $U(1)$-flat connections on $X$.

Let  $A$ be a connection on a Hermitian vector bundle $E$ over $X$.   We have the family $\{  A_{\rho} \}_{\rho \in \hat{X}}$ of connections parametrized by $\hat{X}$.    Here  $A_{\rho} = A \otimes \rho$.  

Let $S^{\pm}$ be the spinor bundles over $X$. Then we have the twisted Dirac operators
\[
   \begin{split}
       & D_{A_{\rho}} : \Gamma(E \otimes L_{\rho} \otimes S^+) \rightarrow \Gamma(E \otimes L_{\rho} \otimes S^-),  \\
       & D_{A_{\rho}}^* : \Gamma(E \otimes L_{\rho} \otimes S^-) \rightarrow \Gamma(E \otimes L_{\rho} \otimes S^+). 
   \end{split}
\]
Here $L_{\rho}$ is the flat line bundle corresponding to $\rho$. 

From now on, we assume  that $A$ is ASD (and hence $A_{\rho}$ is also ASD) and that the following condition holds:
\begin{equation}
   \ker D_{A_{\rho}}^* = 0 \quad (\forall \rho \in \hat{X}). 
\end{equation}

 Applying the Nahm transform in Section \ref{section Nahm commutative} to $\{ A_{\rho} \}_{\rho \in \hat{X}}$, we obtain a connection $\hat{A}$ on the bundle $\hat{E} \rightarrow \hat{T}$. 
More precisely, the bundle
\[
      \coprod_{\rho \in \hat{X}} L^2(S^- \otimes E \otimes L_{\rho}) \rightarrow \hat{X}
\]
does not have a natural trivialization. This means that we can not apply the Nahm transform to $\{  A_{\rho}  \}_{\rho \in \hat{X}}$ directly.  To avoid this issue, we consider the family of ASD connections $\{ A_{\tilde{\rho}} \}_{\tilde{\rho} \in H^1(X ; \R)}$ parameterized by  the universal cover $H^1(X ; \R)$ of $\hat{X}$.  Then the flat line bundle $L_{\tilde{\rho}}$ can be considered to be the pair $(\underline{\C}, \tilde{\rho})$ of the trivial bundle complex line bundle $\underline{\C}$ over $X$ and the flat connection $\tilde{\rho}$.  Since $L_{\tilde{\rho}}$ is trivial as a topological complex line bundle, we can think that the operators $D_{\tilde{\rho}}$ act on the the same space $\Gamma(S^- \otimes E)$.
Therefore we can apply the Nahm transform to this family and we get a connection $\tilde{A}$ on 
\[
            \tilde{E} = \coprod_{\tilde{\rho} \in H^1(X; \R)} \ker D_{A_{\tilde{\rho}}}^*  \rightarrow H^1(X; \R).
\]
Recall that the corresponding covariant derivative $\nabla_{\tilde{A}}$ is given by the formula: 
\begin{equation} \label{eq nabla tilde}
           \nabla_{\tilde{A}} = p \nabla i.
\end{equation}
Here $i$ is the family of inclusions  $i_{\tilde{\rho}} : \ker D_{A_{\tilde{\rho}}} \hookrightarrow L^2(S^- \otimes E )$, 
 $p$ is the family of projections $p_{\tilde{\rho}} : L^2(S^- \otimes E) \rightarrow \ker D_{A_{\tilde{\rho}}}$ and $\nabla$ is the trivial connection. 
If $\tilde{\rho} \in H^1(X;\Z)$, we have
\[
              D_{ A_{\tilde{\rho} }}^* = u^{-1} D_{A_0}^* u,
\]
where $u : X \rightarrow U(1)$ is the harmonic gauge transform corresponding to $\tilde{\rho}$. Hence:
\begin{equation}  \label{eq i p}
      i_{\tilde{\rho}} =  i_0 u,  \quad
      p_{\tilde{\rho}} = u^{-1} p_0.
\end{equation}
By (\ref{eq nabla tilde}) and (\ref{eq i p}), we can see that the connection $\tilde{A}$ naturally  descends to a connection $\hat{A}$ on the bundle:
\[
           \hat{E} = \coprod_{\rho \in \hat{X}} \ker D_{A_{\rho}}^* 
\]
because 
\[
         \hat{E} =  \coprod_{  \tilde{\rho} \in H^1(X;\R) } \ker D^*_{\tilde{\rho}}  \  / \ H^1(X;\Z),
\]
where the action of $H^1(X;\Z)$ is the adjoint of the harmonic gauge transforms
 $u : X \rightarrow U(1)$. 

\begin{theorem}
The connection $\hat{A}$ on $\hat{E}$ is ASD. 
\end{theorem}

\begin{proof}
We will show outline of the proof. 

Put $V := H^1(X;\R) (\cong \R^4)$ and let $\pi : V  \rightarrow \hat{X}$ be the projection.  Then $\hat{A}$ is ASD if and only if $\pi^* \hat{A} \ ( = \tilde{A})$ is ASD.  The Riemannian metric on $X$ naturally induces a metric on $V$. 
Take any complex structure $\tilde{J}$ on $V$ compatible with the metric.  As explained in Section 3.1.3 of \cite{DK}, $\tilde{E} $ has a natural holomorphic structure, and  an infinite dimensional version of Lemma (3.1.20) of \cite{DK} shows that the connection $\tilde{A}$ is compatible with the holomorphic structure. 
Therefore
\[
      F_{ \tilde{A} } \in \Lambda^{1, 1}_{\tilde{J}} \otimes \frak{g}_{\tilde{E}} 
\]
at each point.  By Proposition \ref{prop (1,1) type}, $\tilde{A}$ (and hence $\hat{A}$) is ASD.

\end{proof}

 \section{Higher Nahm transform}
Let $(X,g)$ be a compact Riemannian four manifold, 
and $E \to X$ be a unitary bundle
equipped with a connection $A$. 
Denote $\tilde X$ a $\Gamma$-cover of $X$ with covering group $\Gamma$.
For example, 
$\Gamma= \pi_1(X)$ is the fundamental group and $\tilde X$ is the universal cover. Consider the accociated $C^*_r(\Gamma)$ bundle:
$$\gamma := \tilde{X} \times_{\Gamma} C^*_r(\Gamma).$$
Assume that $X$ is spin, and let $S$ be the spinor bundle.
Denote $S \hat{\otimes} E := S_E$. 
Let us introduce an $L^2$-inner product 
on the smooth sections of $S_E \otimes \gamma$:
$$(f,g) : = \int_X <f(x),g(x)> vol \ \in \ C^*_r(\Gamma)$$
where $< \quad >$ is the Hilbert module inner product,
and denote its completion by $L^2(S_E \otimes \gamma)$.

A connection  $A$ on $E$ induces the twisted connection on $E \otimes \gamma$ by:
$$\nabla_A(f \otimes\sigma) = \nabla_A(f ) \otimes \sigma +  f \otimes d \sigma.$$
Hence $A$ induces a Dirac operator $D_{A}$ on
$S_E  \otimes \gamma$:
$$D_{A}^+ : L^2(X; S^+_E \otimes \gamma)
 \to L^2(X; S^-_E \otimes \gamma).$$
 Both ker $D_A$ and  coker $D_A$  consists of finitely generated projective $C^*_r(\Gamma)$  modules after compact perturbation.
  The higher index is defined as their formal difference:
$$\text{ind } D_A := [\text{ ker } D_A] - [\text{ coker } D_A]
 \in K_0(C^*_r(\Gamma)).$$

 Suppose ker $D_A^+ =0$, and 
consider the finitely generated projective 
$C^*_r(\Gamma)$ module:
$$\cE_A := \text{ coker } D_A ^+ =\text{ ker } D_A^-
\subset 
L^2(X; S^-_E \otimes \gamma).$$

\begin{lemma}\label{proj}
Suppose ker $D_A^+ =0$.
Then there is the $C^*_r(\Gamma)$ module projection:
$$P: L^2(X; S^-_E \otimes \gamma) \to
\cE_A
$$
\end{lemma}

\begin{proof}
Notice that  ker $ (D_A^-)^* =$ ker $D_A^+ =0$ by the assumption.
Then we consider the bounded operator:
$$ P: = \text{ id} - (D_A^-)^* (D_A^- (D_A^-)^*)^{-1} D_A^-$$
on $L^2(X; S^-_E \otimes \gamma)$.
It is easy to check that this satisfies the required properties.
\end{proof}

\begin{remark}
Notice that in general the projection does not exists
for a Hilbert $C^*_r(\Gamma)$ module embedding
$\cE \hookrightarrow \cH$.

Actually $\cE \oplus \cE^{\perp}$ do not coincide with
$\cH$ in general.
\end{remark}

\subsection{Quatized calculus}

We recall a notion of 
connection on a finitely generated projective module associated to a spectral triple,
in non commutative geometry.

Let $(\cA, \cH, D)$ be a spectral triple, where $\cA$ is a unital $*$-algebra represented in a Hilbert space $\cH$ as $\pi: \mathcal{A} \to \mathcal{L}(\cH)$,
  where $\mathcal{L}(\cH)$ is the set of all bounded operators on $\cH$.
   $D$ is an unbounded operator on $\cH$ such that
$[D, a]\in\cL(\cH)$ is bounded for any $
a \in \cA$, and $(1+D^2)^{-1}\in\cK(\cH)$ is compact.
  Suppose  Ker $D=0$ and define: 
  $$\hat{d}(a) :=  i[F,a] \in \mathcal{L}(H)$$
  where $F = \frac{D}{|D|}$.
  Let us consider the linear space:
  $$\hat{\Omega}^*_F : = \ \text{ span } \{ \  a^0    \hat{d}a^1 \dots \hat{d}a^q \ \}$$
   spanned by vectors of the form
$a^0  \hat{d}a^1 \dots \hat{d}a^q$
  for $q \geq 0$, where  $a^i \in \mathcal{A}$.
It  admits a structure of  graded algebra
  $\hat{\Omega}_F = \oplus^q \ \hat{\Omega}_F^q$.
  We call an element in this space as a non commutative differential form.

  \begin{lemma}\cite{Connes94}
  It satisfies the following two properties:
  \begin{align*}
  & (1) \ \ \hat{d}^2=0, \\
  & (2) \ \ \hat{d}(a_1a_2) = (\hat{d}a_1 ) a_2 +  a_1 \hat{d} a_2.
  \end{align*}
  \end{lemma}

Let
$\mathcal{E}$ be a finitely generated projective module over $\cA$. 
A {\em connection} $A$ on $\mathcal{E}$ is a linear map with a derivation property:
\begin{align*}
& \nabla_A: \mathcal{E}\rightarrow \mathcal{E}\otimes_{\cA}\hat\Omega^1_D, \\
& \nabla_A(ea)=(\nabla_A e)a+e\hat{d}a \qquad e\in \mathcal{E}, a\in \cA. 
\end{align*}

\begin{remark}
If $\cE=\Gamma(S)$ denotes the space of smooth sections of 
the spinor bundle $S$ over a closed spin manifold $X$,
 then $\cE$ is a finitely
generated  projective module over $\cA=C^{\infty}(X)$. In this case a canonical choice of spectral triple is the Dirac spectral triple $(C^{\infty}(X), L^2(X, S), D)$ where $D$ is the Dirac operator on the spinor bundle.
It is known that for $a\in C^{\infty}(X),$  the equality holds:
\[
[D, a]=c(da)
\]
and hence, 
one recovers the classical $1$-form $da\in C^{\infty}(X, T^*X)$ 
via the Clifford multiplication: 
\[
c: TX\rightarrow  \End S.
\]
\end{remark}

Assume that $C^*_r(\Gamma)$ admits a dense subalgebra $C^{\infty}(\Gamma)$ containing $\C\Gamma$ closed under holomorphic functional calculus.   
Let  $\cA = C^{\infty}(\Gamma)$, and 
$H \otimes C^{\infty}(\Gamma)$ be the trivial $C^{\infty}(\Gamma)$
module equiped with the trivial connection
$\hat{d}$  on $\mathcal{H} $
given by:
$$\hat{d}(h \otimes u) := h \otimes \hat{d}(u) \in H \otimes \hat{\Omega}^1 = 
 (H \otimes C^{\infty}(\Gamma) )
 \otimes_{C^{\infty}(\Gamma)} \hat{\Omega}^1_F.$$
 
Here $H\otimes C^{\infty}(\Gamma)$ is viewed as a dense subspace of its $C^*$-module completion $H \otimes C^*_r(\Gamma)$. 
 
 It is well known as Kasparov's stabilization that there exists a Hilbert $C^*_r(\Gamma)$
 module isomorphism:
 $$L^2(X; S^-_E \otimes \gamma) \oplus (H \otimes C^*_r(\Gamma) )
 \cong H \otimes C^*_r(\Gamma) .$$
 
 Let us fix this isomorphism, and let:
 $$P': H \otimes C^*_r(\Gamma)  \to L^2(X; S^-_E \otimes \gamma)$$
 be the projection.
 
 Recall lemma \ref{proj}.
 Composed with another projection
 $P =P_A: L^2(X; S^-_E \otimes \gamma) \to \cE_A$, 
 one obtains the projection:
 $$Q_A: = P_A \circ P': \ H \otimes C^*_r(\Gamma)  \to \cE_A.$$
Denote by $\cE_A^{\infty}$ the dense subspace of $\cE_{A}$ obtained by intersecting $\cE_A$ with smooth sections $C^{\infty}(X,  \tilde S_E \otimes_{\Gamma}C^{\infty}(\Gamma))$ in $L^2(X, S_E\otimes\gamma).$
The trivial connection on $H\otimes C^{\infty}(\Gamma)$ induces  the connection on $\mathcal{E}_A^{\infty}$
by composition:
$$\hat{d}_A: = Q_A \circ \hat{d}: \ \cE_A^{\infty} \to 
\mathcal{E}_A^{\infty} \otimes_{\cA}\hat{\Omega}^1_F.$$

\begin{definition}
A  higher Nahm transform is given by:
$$(E,A) \to (\mathcal{E}_A^{\infty}, \hat{d}_A).$$
\end{definition}
Notice that the right hand side is a non commutative object and 
does not involve space.

 \subsection{Dixmier trace and Connes-Yang-Mills functional}

In~\cite{ConnesYM}, Connes  made use of Dixmier trace to study the Yang-Mills  functional. We recall the definitions and constructions. 

For every positive element $A \in\mathcal{K}(\mathcal{H})$, put:
$$\mu_n=\inf_{T\in R_n}\|A-T\|$$ 
where $R_n=\{Q\in\cL(\cH): \mathrm{rank}(Q)\le n\}$,
and 
$$\delta_N (A) = \sup \ \{ \  \Tr (AP) ; \ \rk (P) \le N \ \}$$
 where the supremum is over all projections $P$ whose ranks are no more than $N$. 
 Equivalently, if $\lambda_i\ge 0$ is a decreasing sequence of eigenvalues of $A$, then 
 \[
 \delta_N(A)=\sum_{i=1}^N \lambda_i\qquad \mu_n(A)=\lambda_n.
 \]
 
For $T\in\cK(\cH)$, let us define a norm: 
\[
\|T\|_{(1, \infty)}:=\sup_{N}\frac{1}{\log(1+N)}\sum_{n=1}^N\mu_n(T).
\]
This gives rise to the Banach ideal of $\cK(\cH):$
\[
\cL^{(1, \infty)}(\cH)=\{T\in\cK(\cH): \|T\|_{(1, \infty)}<\infty\}.
\]
Denote its positive cone by: 
\[
\cL^{(1, \infty)}_+(\cH)=\{T\in\cL^{(1, \infty)}(\cH): T\ge0\}.
\] 
This is the domain of the Dixmier trace.  
 
For every linear form $\omega$ on $l^{\infty}$, denote 
\[
\omega-\lim(\{a_n\}):=\omega(\{a_n\}).
\] 
  Let $\omega$ be a state (positive linear functional with norm $1$) on $l^{\infty}$ satisfying the following properties:
\begin{itemize}
\item $\omega-\lim(\{a_n\})\ge0$ if $a_n\ge 0;$
\item $\omega-\lim(\{a_n\})=\lim a_n$ if $\{a_n\}$ is convergent; 
\item $\omega$ is invariant under dilation $D_2$, i.e., 
\begin{equation}
\label{eq:dialation}
 \omega-\lim(\{a_n\})=\omega-\lim(a_1, a_1, a_2, a_2, \ldots)\qquad \forall \{a_n\}\in l^{\infty}.
\end{equation} 
\end{itemize}

\begin{definition}
The Dixmier trace $\Tr_{\omega}: \cL^{(1, \infty)}_+(\cH)\rightarrow[0, \infty)$ is defined by:
\begin{align*}
\Tr_{\omega} (A)  & = \omega- \lim_{N \to \infty} \  
\frac{1}{\log(N+1)} \ \delta_N(A) \\
& =\omega- \lim_{N \to \infty} \  
\frac{1}{\log(N+1)} \sum_{n=1}^N\mu_n(A).
\end{align*}
We extend it to $\Tr_{\omega}: \cL^{(1, \infty)}(\cH)\rightarrow \C.$
\end{definition} 

 Consider a complex vector bundle $E$ on a closed manifold $X$ of dimension $n$ and the Hilbert space $\cH=L^2(X, E)$. 

For $m\in\Z$, denote by $\Psi^m(X, E)$ the space of pseudo differential operators over $X$ of order $m.$
For $P\in\Psi^{-n}(X, E)$, denote by $\sigma_P\in C^{\infty}(S^*X, \End E)$ its principal symbol. 
Define the Wodzicki residue of $P$ by: 
\[
\Res(P)=\frac{1}{(2\pi)^n}\int_{S^*X}\tr_E\sigma_{-n}(P)d\nu.
\]
Here, $d\nu$ is the volume element defined by 
$\frac{(-1)^{\frac{n(n+1)}{2}}}{(n-1)!}\alpha\wedge(d\alpha)^{\wedge(n-1)}$, where $\alpha =\sum_{i}\xi_idx_i$ in local coordinate, is the canonical $1$-form on $T^*X$,
and $S^*X$ is the unit sphere bundle.

Below is the Connes' trace formula:
\begin{theorem}~\cite{ConnesYM}
Let $E \to X$  and $\cH$ be as above. 
Let $P\in\Psi(X, E)$ be a pseudo differential operator of order $-n$. Then $P\in\cL^{(1, \infty)}(\cH)$ and the equality holds:
\[
\Tr_{\omega}(P)=\frac{1}{n}\Res(P).
\]
\end{theorem}

Let $X$ be a closed oriented spin $4$-manifold. Let $E$ be a finite dimensional complex vector bundle over $X$. Consider now the Hilbert space $\cH$ of the form 
\[
\cH=\cE\otimes_{\cA}\cH_0=L^2(X, E\otimes S) 
=L^2(X, S_E) 
\]
where $\cE=C^{\infty}(E)$, $\cA=C^{\infty}(X)$ and $\cH_0=L^2(X, S).$
Here $\cH$ is again a module over $\cA$ because $\cA$ is commutative. 

Let $A_c: \cE\rightarrow\cE\otimes_{\cA}C^{\infty}(X, T^*X)$ be a connection on $\cE$. 
These data give
 the Dirac spectral triple $(\cA, \cH_0, D)$,
 where $D=D_{A_c}$ acting on sections of $S_E$.
  Then one obtains a quatized connection replacing $da$ for $a\in C^{\infty}(X)$ by $[D|D|^{-1}, a]\in\cK(\cH_0)$,
  when ker $D=0$. 
It is a pseudo differential operator of order $-1$.

Denote $F=D|D|^{-1}$ and $\Omega^1_F$ be the span of quatized $1$ forms $a[F, b]\in\cL(\cH_0)$ for $a, b\in \cA$. 

$A: \cE \to \cE\otimes_{\cA} \Omega_F^1$ induces the quatized covariant derivative:
$$A: \cE\otimes_{\cA} \Omega_F^1 \to \cE\otimes_{\cA} \Omega_F^2$$
given by: 
$$A(e \otimes a[F,b] ) : = A(e) \otimes a[F,b] + e \otimes [F,a][F,b].$$
The quatized curvature is defined by the composition:
$$\theta: = A^2 : \cE \to \cE\otimes_{\cA} \Omega_F^2.$$
It follows from the formula:
$$\theta(ea) = \theta(e)a$$
 that $\theta$ can be regarded as 
an element in $\cL(\cH)$, induced by the composition:
$$\theta: \cE \otimes_{\cA} \cH_0 \to  \cE \otimes_{\cA} \cL(\cH_0) \otimes_{\cA} \cH_0
\to  \cE \otimes_{\cA} \cH_0.$$

There is a well defined surjective map $c: \Omega^1_F\rightarrow\Omega^1_{\cA}$ such that $c([F,a])=da.$
Denote by $\theta_c$ and $\theta$ the curvature of $A_c$ and the quatized connection $A$ respectively. 

\begin{theorem}[Connes \cite{ConnesYM}]
Suppose $X$ is of four dimension. 
The square of the quatized curvature $\theta^2\in\cL^{(1, \infty)}(\cH)$ gives rise to a positive functional:
\[
A\rightarrow \inf_{c(A)=A_c}\Tr_{\omega}(\theta^2)
\]
independent of $\omega$. 
The function coincide with the classical Yang-Mills functional up to a constant. 
\end{theorem}

\subsection{Pre Yang-Mills  functional}

We propose a model for pre Yang-Mills  functional using Dixmier trace for a finitely generated projective module. If the algebra involved is non commutative, then it is in general difficult to formulate an analogy to the Connes-Yang-Mills  functional in the commutative case. 

Let $(\cA, \cH_0, D)$ be a spectral triple. 
Assume the spectral triple is $n$-summable, i.e., $|D|^{-n}\in\cL^{(1, \infty)}(\cH_0)$. In our setting, $\cA=C^{\infty}(\Gamma)$ for $\Gamma$, a covering group  of a closed $4$-manifold $X$ and we will assume $n=4.$
Let $F=D|D|^{-1}\in\cL(\cH_0)$ 
and 
$\Omega^1_F(\cA)=\{a[F, b]| a, b\in \cA\}\subset\cL(\cH_0)$
be as before.

Let $\cE$ be a finitely generated right projective $\cA$-module, with a connection $\nabla^{\cE}: \cE\rightarrow\cE\otimes_{\cA}\Omega^1_F(\cA).$
By definition $\nabla^{\cE}$ satisfies the equality:
\[
\nabla^{\cE}(ea)=(\nabla^{\cE}e)a+e[F, a]\qquad e\in\cE, a\in\cA.
\]
We define its composition $\rfloor \circ (\nabla^{\cE})^2: \cE\otimes_{\cA}\cH_0\rightarrow\cE\otimes_{\cA}\cH_0$ with the contraction $\rfloor$ by: 
\[
\rfloor \circ (\nabla^{\cE})^2(\xi\otimes\zeta)=\sum\xi_{\alpha}\otimes \omega_{\alpha}(\zeta)
\]
when $\nabla^{\cE}\xi=\sum\xi_{\alpha}\otimes\omega_{\alpha}$ and $\xi_{\alpha}\in\cE, \omega_{\alpha}\in\Omega^2_F(\cA).$






 
 It can be checked that $[F, b]^n\in\cL^{(1, \infty)}(\cH)$ for $b\in\cA.$ 
 Assume $n=4.$
 We therefore obtain that the quatized curvature 
 \[
 \Theta=(\nabla^{\cE})^2\in\cL(\cH)
 \]
 admits a Dixmier trace:
 \[
 \Tr_{\omega}(\Theta^2)<\infty.
 \]
 See page 680 of~\cite{ConnesYM} and~\cite{ConnesNDG}. 
 We propose that a pre Connes-Yang-Mills  functional
  is given by a map: 
\[
 \nabla^{\cE}  \rightarrow \Tr_{\omega}(\Theta^2)
\]
 for every connection $\nabla^{\cE}$ on $\cE$, 
where $\Theta=(\nabla^{\cE})^2$ is the quatized curvature. 
 
\begin{remark}
Unlike the commutative case, we 
have no  space, and hence
do not have the surjective map $c: \Omega^1_F\rightarrow\Omega^1_{\cA}$ mapping $i[F, a]$ to $da$ as in~\cite{ConnesYM}.
 So we can not take minimum of the right hand side. Hence the name.
\end{remark}

 \section{Dixmier $\Gamma$ trace}

 Let $\mathcal{E}$ be a Hilbert $\mathcal{A}$ module, and 
$\cL(\mathcal{E})$ be the set of bounded $\mathcal{A}$-linear endomorphisms. 
Let $\mathcal{K}(\mathcal{E}) \subset \cL(\mathcal{E})$
 be the set of compact endomorphisms given by the norm closure of finite rank projections. 
  $\cE$ is reduced to a Hilbert space $\cH$ when 
$\mathcal{A}=\mathbb{C}.$

\subsection{Dixmier $\Gamma$-trace}
\label{sec:DixmierGamma-trace}

Dixmier $\Gamma$-trace is an instance of type II non commutative geometry. 
In~\cite{BF}, Dixmier trace was introduced on
  von Neumann algebras. 
 Our specific context here is to reformulate this.
Let $\mathcal{A}= C^*_r(\Gamma)$ be the reduced group $C^*$-algebra.
Denote by the canonical von Neumann trace:
\begin{equation}
\label{eq:vNtrace}
\tr: C^*_r(\Gamma)\rightarrow\C
\end{equation}
 determined by: 
\[
\sum_{\gamma\in \Gamma}a_{\gamma}\gamma\mapsto a_e.
\]
Let $\cE$ be a Hilbert $C^*_r(\Gamma)$-module. Denote
 by $\cK(\cE)$ the algebra of $\cA$-linear compact endomorphisms on $\cE$,
  and by $\cK(\cE)_{+}$ the subset of positive elements. 
  
  Kasparov's stabilization implies that there is a Hilbert module isomorphism
$\cE \oplus H \otimes C^*_r(\Gamma) \  \cong \ \cE$,
where $H$ is a separable Hilbert space.
Passing through this, a
compact endomorphism
 on a Hilbert module can be represented as an infinite matrix with entries in $C^*_r(\Gamma)$. 
Let: \[
\tr_{\Gamma}=\tr\circ\Tr: \cK(\cE)_+\rightarrow[0, \infty]
\] 
be the trace induced by~(\ref{eq:vNtrace}) , and extend it
to $\tr: \cK(\cE) \to \C$,
where $\Tr$ is the matrix trace. 
When $\tr_{\Gamma}(A)<\infty$, then $A$ is known as a $\Gamma$-trace class operator. 
Note that when $\Gamma$ is trivial, $A$ is a compact operator in the usual sense and $\tr_{\Gamma}(A)$ is the operator trace of $A.$ 

For $t\ge 0$ and an operator $A\in\cK(\cE)_+$, recall that the generalised singular value function: 
\[
\mu(A): [0,\infty)\rightarrow[0, \infty)
\] 
is given by: 
\[
\mu_t(A)=\inf\{s\ge 0: \tr_{\Gamma}(\chi_{(s, \infty)}(A))\le t\}
\] 
where $\chi$ is the characteristic function (see ~\cite{F, FK, GT}).

When $\Gamma$ is trivial, a compact positive operator $A$ on a Hilbert space  has  point spectrum only.
 In the case
 $\mu_n(A)$ is the $(n+1)$-th largest eigenvalue of $A$, 
 taken into account of multiplicity for $n\in\{0\}\cup\N$.
 More precisely we have:

\begin{lemma}\label{dim}
Let $\{\lambda_i\}_{i=0}^{\infty}$ be the set of eigenvalues of a compact positive operator $A$ on a Hilbert space, counted with multiplicity, decreasing to $0$ as $i\to\infty.$ 
Then: 
\[
\mu_t(A)=\lambda_i \qquad
\text{where} \qquad \sum_{j=0}^{i-1}\dim E_{\lambda_j}\le t< \sum_{j=0}^{i}\dim E_{\lambda_j}.
\]
Here $E_{\lambda_j}$ stands for the eigenspace of $\lambda_j.$
\end{lemma}

\begin{proof}
By assumption, $A=\oplus_j\lambda_i E_i$. Then: 
\[
\chi_{(s, \infty)}(A)=\oplus_{\lambda_i>s} P_{E_{\lambda_i}}\qquad \text{and} \qquad \Tr(\chi_{(s, \infty)}(A))=\sum_{\lambda_i>s}\dim E_{\lambda_i}. 
\] 
By definition, we have: 
\[
\mu_t(A)=\inf\{s\ge0: \sum_{\lambda_j>s}\dim E_{\lambda_j}\le t\}
\]
which gives rise to the statement.
\end{proof}

Suppose  $P$ is a projection in $\cL(\mathcal{E})$ of $\Gamma$-trace class.
Then the $\Gamma$-rank or $\Gamma$-dimension is given by: 
\[
\rk_{\Gamma}P:=\tr_{\Gamma}(P) \ \in \ \mathbb{R}.
\] 
The rank does not have to be an integer.
For every $r>0$ and  a positive element $A\in\mathcal{K}(\mathcal{E})_+$,
let us  put:
 $$\delta_r^{\Gamma} (A) :=  \  
 \sup \ \{ \  \tr_{\Gamma} (AP) ; \ \rk_{\Gamma} (P) \le r \ \}$$
where the supremum is taken 
over all  projections $P$
with $\Gamma$-rank  at most $r$
 in $\mathcal{K}(\mathcal{E}).$
Observe the equality:
\[
\delta^{\Gamma}_r(A)=\int_0^r\mu_t(A)dt.
\] 
See Lemma $A.2$ in~\cite{BF} for more details.




For $A\in\cK(\cE)_+$, let us introduce a norm:  
\begin{align*}
\|A\|_{(1, \infty)} & :=\sup_{t>0}\frac{1}{\log(1+t)}\delta_t^{\Gamma}(A) 
\\
& =\sup_{t>0}\frac{1}{\log(1+t)}\int_0^t\mu_s(A)ds.
\end{align*}
We then have a subset of $\cK(\cE)$:
\[
\cL^{(1, \infty)}_+(\cE)=\{ \ T\in\cK(\cE)_+ : \|T\|_{(1, \infty)}<\infty \ \}
\]
and an ideal of $\cK(\cE)$:
$$\cL^{(1, \infty)}(\cE)=\{ \ T\in\cK(\cE): 
Re(T)_{\pm}, Im(T)_{\pm}\in\cL_{+}^{(1, \infty)}(\cE) \ \}.$$ 

A state $\omega$ defining the Dixmier $\Gamma$ trace 
 exists by the following lemma. See for example~\cite{BF, S}.

\begin{lemma}
There exists a state $\omega$ on $L^{\infty}(0, \infty)$ satisfying the following conditions:
\begin{itemize}
\item $\omega(C_0(0, \infty))=0;$
\item If $f$ is real-valued in $L^{\infty}(0, \infty)$, then 
\[ 
ess\liminf_{t\to\infty} f(t)\le\omega(f)\le ess\limsup_{t\to\infty}f(t);
\]
\item $\omega(f)=0$ vanishes, if
 the essential support of $f$ is compact;
\item  For any $a>0$, the equality:
$$\omega=\omega\circ D_a$$
holds, where 
 $D_a$ is  the dilation 
$
D_a(f)(x)=f(\frac{x}{a}) $
for $ f\in L^{\infty}(0, \infty)$.
\end{itemize}
\end{lemma}

Again, denote the evaluation of $\omega$ at $f\in L^{\infty}(0, \infty)$ by: 
\[
\omega-\lim_{r\to\infty}f(r):= \ \omega(f).
\]
Let $A$ be a positive operator in $\mathcal{K}(\mathcal{E})$.
The absolute Dixmier $\Gamma$   trace on $\mathcal{L}^{(1,\infty)} (\mathcal{E})$
 is given by:
$$\Tr_{\omega}^{\Gamma} (A) :=  \ \omega- \lim_{r \to \infty} \  
\frac{1}{\log (r+1)} \ \delta_r^{\Gamma}(A).$$

\begin{definition}
Let $A$ be a self adjoint operator.
The Dixmier $\Gamma$   trace:
 $$\Tr_{\omega}^{\Gamma}: \cL^{(1, \infty)}(\cE)\rightarrow\C$$ is defined by:
$$\Tr_{\omega}^{\Gamma} (A) = 
\Tr_{\omega}^{\Gamma} (A_+) - \Tr_{\omega}^{\Gamma} (A_-)
.$$
\end{definition}
For  arbitrary $A$, it is given by:
$$\Tr_{\omega}^{\Gamma} (A) = 
\Tr_{\omega}^{\Gamma} (\text{Re }  A) + i  \Tr_{\omega}^{\Gamma} (\text{im } A).
$$

\section{Higher Yang-Mills instantons}

Let $X$ be an oriented closed smooth $4$-manifold with a spin structure. 
Let $\tilde X$ be a $\Gamma$-cover of $X$ with covering group $\Gamma$.
For example, 
$\Gamma= \pi_1(X)$ is the fundamental group and $\tilde X$ is the universal cover.
Let $\gamma=\tilde X\times_{\Gamma}C^*_r(\Gamma)\rightarrow X$ be the Mischenko-Fomenko
 line bundle with flat connection $\nabla^{\gamma}$. 
Let $S \to X$ be
 the spinor bundle and denote its lift by $\tilde{S}$  
 on the universal covering space  $\tilde X$. 
Let $E\rightarrow X$ be a unitary bundle equipped with a unitary
connection $A=\nabla^E.$ 

Consider:
\begin{align*}
 & \cE_0=C^{\infty}(X, \tilde E\times_{\Gamma}C^*_r(\Gamma))
 = C^{\infty}(X,  E\otimes \gamma),
  \\
 & \cA=C^{\infty}(X), \\
 & \cH=L^2(X, S) 
 =L^2(X,  S^+)
 \oplus
 L^2(X,  S^-)
 \end{align*}
 where $\cE_0$ is 
a finitely generated 
$(\cA, C^*_r(\Gamma))$ bi-module
 in the sense that it is 
 a left module over $C^{\infty}(X)$ and is a right module over $C^*_r(\Gamma).$ 
 $\cH$ is 
 a $\cA$ bi-module.
 
 Therefore, we form the Hilbert $C^*_r(\Gamma)$-module: 
\begin{align*}
\cE=\cE_0\otimes_{\cA}\cH & =C^{\infty}(X, E \otimes \gamma)
\otimes_{C^{\infty}(X)} 
L^2(X,  S)\\ 
& = L^{2}(X, E\otimes  S\otimes
\gamma).
\end{align*} 
Let $\Omega^1_{\cA}$ be the space of smooth $1$-forms of $\cA$:
\[
\Omega^1_{\cA}=\{ \ \sum adb|a, b\in\cA \ \}=C^{\infty}(X, T^* X).
\]

Let $D: \cH\rightarrow\cH$ be the Dirac operator on $S$. 
Locally $D$ has the form $\sum_{i}r(e_i)\nabla^{\cH}_{e_i}$, where 
$\{e_i\}_i$ is a local orthonormal frame and
$r(e_i)$ is the Clifford multiplication by $e_i$.
$\nabla^{\cH}$ is the  spin  connection so that:
\[
[\nabla^{\cH}_{e_i}, r(\alpha)]=r(\nabla_{e_i}\alpha)\qquad \alpha\in\Omega^1_{\cA}
\]
holds.
Then it can be checked the equality:
\[
[D, f]=r(d f ) \qquad f\in \cA.
\]
Observe that a one form in $\Omega_{\cA}^1$ acts on $\cH$ as a 
bounded operator
via the inclusion:
\[
r: \Omega_{\cA}^1\hookrightarrow \cL(\cH), \qquad
r(adb)=a[D, b].
\]

The space of connections $A=\nabla^{\cE_0}$ on $\cE_0$ 
forms an affine space over: 
\[
C^{\infty}(X, T^*X\otimes E\otimes_{Ad\Gamma}\End \gamma)
\]
where $\Gamma$ acts on $\End \gamma$ by conjugation.

Assume that ker $D^+=0$ holds, and let:
 \[
F=D^+|D^+|^{-1}
\] 
be the unitary part  of $D$. Introduce the quantized differential by:
\[
\hat d a:= \ i [F, a] \qquad a\in \cA
\]    
and denote by $\Omega^1_F$ the space of quantized $1$-forms:
\[
\Omega^1_F=\{ \ \sum a\hat d b\in \cL(\cH)|a, b\in\cA \ \}.
\]
Similarly, one has the quantized differential forms of degree $k$:
\[
\Omega^k_F=\{ \ \sum a^0\hat da^1\cdots \hat da^k\in\cL(\cH)| a^j\in\cA  \ \}
\]
and their differentials:
\[
\alpha\in\Omega^k_F \ \rightarrow 
\ \hat d\alpha:= \ i(F\alpha-(-1)^k\alpha F) \ \in \ \Omega^{k+1}_F.
\] 

\begin{lemma}\cite{ConnesYM}
\label{lem:bimod.c}
$(1)$
There exists a unique bimodule linear map
which satisfies the property:
\[
c: \Omega^1_F\rightarrow \Omega^1_{\cA}\qquad \hat d a\mapsto da.
\]
This map is surjective and the image of self adjoint elements of $\Omega^1_F$ 
 are the real forms.
 
 $(2)$ 
 For $\alpha\in\Omega^1_{F},$ we have 
\[
Ac(\hat d \alpha )=d[c(\alpha)]
\]
where $A$ stands for the projection onto the anti-symmetric tensor.
\end{lemma}

\begin{remark}
\label{rmk:Acd}
Recall~\cite{ConnesYM} the meaning of $c$ on $\Omega^2_F$
 in  lemma \ref{lem:bimod.c} 
 and its relaltion to the principal symbol. 
The principal symbol of the pseudo differential operator $\hat d a (a\in\cA)$ of order $-1$
 is given by:
\[
\sigma_{-1}(\hat d a)=r(da-\langle\xi, da\rangle \xi)\in C^{\infty}(S^*X, \Cl(T^*X)).
\]
Notice that the assignment:
\[
da\mapsto da-\langle\xi, da\rangle\xi 
\]
induces a linear  injection:
\[
C^{\infty}(X, T^*X) \hookrightarrow
C^{\infty}(S^*X, \Cl(T^*X)).
\]
In particular $\sigma_{-1}(\hat d a)$ and hence $\hat{d}a$
uniquely determines $da$.
This gives a well defined linear map $c$ on $\hat da\in\Omega^1_F$
with 
$c(\hat da)=da$, 
which 
coincides with $\sigma_{-1}(\hat da)$   up to this identification.    

Then Connes defined $c: \Omega^2_F\rightarrow C^{\infty}(X, (T^*X)\otimes(T^*X))
$ by the image of $\sigma_{-2}$ in $C^{\infty}(S^*X, \Cl(T^*X))$
using the homomorphism:
\begin{equation}
\label{eq:pi}
\pi: C^{\infty}(X, (T^*X)\otimes (T^*X))\rightarrow C^{\infty}(S^*X, \Cl(T^*X))
\end{equation}
determined by $\eta_i\otimes \eta_2 \mapsto r(\eta_1-\langle\xi, \eta_1\rangle)r(\eta_2-\langle\xi, \eta_2\rangle)$ for $\eta_1, \eta_2\in T^*_xX$ and $\xi\in S_x^*X$.
\end{remark}

A q-connection $\nabla^{\cE_0}$ on $\cE_0$ is  given by a linear map: 
\[
\nabla^{\cE_0}: \cE_0\rightarrow\cE_0\otimes_{\cA}\Omega^1_F
\]
such that
$
\nabla^{\cE_0}(\xi\cdot x)=(\nabla^{\cE_0}\xi)\cdot x+\xi\otimes \hat d x
$ holds
for $\xi\in\cE_0$ and $x\in\cA.$
Its 
curvature:
$$\Theta:=(\nabla^{\cE_0})^2: \cE_0
\rightarrow \cE_0 \otimes_{\cA}\Omega^2_{F}
$$
is an $\cA$ module. Hence it induces an endomorphism:
$$\Theta: \cE = \cE_0 \otimes_{\cA} \cH \to \cE$$
by contraction.

Every $q$-connection $\nabla^{\cE_0}: \cE_0\rightarrow\cE_0\otimes_{\cA}\Omega^1_F$ 
determines uniquely a classical connection $\nabla_c^{\cE_0}: \cE_0\rightarrow\cE_0\otimes_{\cA}\Omega^1_{\cA}$ by composition with the bimodule map: 
\[
\nabla_c^{\cE_0}:=(1\otimes c)\circ\nabla^{\cE_0}.
\]

Denote the von Neumann trace by:
 $$\tau(T)=\langle T\delta_e, \delta_e\rangle$$ 
 for  an element $T$ of $\End_{C^*_{r}(\Gamma)}(\gamma)$.
 
\begin{lemma}
\label{lem:Acdc}
The curvature $\Theta_c$ associated to the connection: 
\begin{align*}
\nabla_c^{\cE_0}:
 C^{\infty} & (X,  E\otimes \gamma)   \to
  C^{\infty}(X, T^*X\otimes E\otimes \gamma)
\end{align*}
is the antisymmetric part of the curvature $\Theta$ associated to the $q$-connection $\nabla^{\cE_0}$: 
\[
\Theta_c:=(\nabla_c^{\cE_0})^2=Ac(\Theta).
\]
Here $A$ acts on $T^*X\otimes T^*X$ and does nothing to $\End_{C^*_r\Gamma}(\gamma).$
\end{lemma}

\begin{proof}
The proof essentially follows from Lemma~\ref{lem:bimod.c}.
We only need to verify the equality for a component of $\nabla^{\cE}$ of the form $\omega=a\hat d b\in \Omega^1_{F}$. The quatized curvature $\Theta$ can be represented by: 
\[
d\omega+\omega\otimes \omega=\hat d a\hat d b+a^2\hat d b\hat d b. 
\]
The antisymmetric part of $c(\Theta)$ will have the form: 
\[
Ac(\hat d a\hat d b+a^2\hat d b\hat d b)=A(da\otimes db)+A(a^2db\otimes db)=da\wedge db.
\]
On the other hand, the component of $\Theta_c$ is: 
\[
d[c(\omega)]=d(adb)=da\wedge d b. 
\]
The lemma is then proved. 
\end{proof}


Let $A_c$ be a unitary connection on $E$.
It induces  a connection $\mathbb{A}_c$ on $E \otimes \gamma$:
\[
\mathbb{A}_c: C^{\infty}(X, E\otimes\gamma)\rightarrow C^{\infty}(X, T^*X\otimes
E\otimes\gamma  )
\] 
 Replacing $1$-forms  appeared in $A_c$, $\mathbb{A}_c$ 
 respectively by $[F, \alpha]$, we obtain two quantized connections:
\begin{align*}
 &A \text{ on  } C^{\infty}(X,E), \qquad 
 \mathbb{A} \text{ on  }  C^{\infty}(X,E \otimes \gamma )=
 \cE_0.
\end{align*}
Both the curvatures  $\theta=A^2$ and $ \Theta=\mathbb{A}^2$ 
are pseudo differential operators of order $-2$.  

$\Theta: L^2(X, E\otimes S\otimes\gamma)\rightarrow L^2(X, E\otimes S\otimes\gamma)$ is a pseudodifferential operator of order $-2$. Its principal symbol $\sigma(\Theta)$ is an element of 
\[
C^{\infty}(S^*X, \End(\pi^*(E\otimes S\otimes \gamma)))=C^{\infty}(S^*X, \mathrm{Cl}(T^*X)\otimes\End\pi^*(E\otimes\gamma)).
\] 
Thus, this gives rise to 
\[
c(\Theta)\in C^{\infty}(X, T^*X\otimes T^*X\otimes\End E\otimes\End_{C^*_r\Gamma}(\gamma)).
\]
Let $\Theta_{ij}$ be the decomposition of $\Theta$ according to $\End E$. So 
\[
c(\Theta_{ij})\in C^{\infty}(X, T^*X\otimes T^*X\otimes\End_{C^*_r\Gamma}(\gamma)).
\]
Denote by 
\[
\|c(\Theta_{ij})\|\in C^{\infty}(X, \End_{C^*_r\Gamma}(\gamma))
\] 
the $C^{\infty}(X, \End_{C^*_r\Gamma}(\gamma))$-valued norm.  

\begin{lemma}
\label{lem}
 Let $X$ be a closed, oriented, spin, smooth $4$-manifold with $b_1=4$. Let $\Gamma=\Z^4$ be a covering group of $X$. (Existance of a $\Z^4$-cover is ensured by $b_1=4$.)  
Then the formula:
\[
\Tr_{\omega}^{\Gamma}(\Theta^2)=\sum_{i,j}\int_{X} \tau(\|c(\Theta_{ij})\|^2)
\]
holds,
where $\Theta_{ij}$ is  local representation 
by the  $\End_{C^*_r\Gamma}(\gamma)$ valued matrix.
\end{lemma}
  
\begin{proof}
The quatized curvature $\Theta$ is a pseudo differential operator of order $-2$ and is self adjoint, i.e., $\Theta_{ij}^*=\Theta_{ji}$ (similar argument as in~\cite{ConnesYM}).
Denote $P=\Theta^2=\Theta\Theta^*$. Then
\[
\sigma_{-4}(P)=\sigma_{-2}(\Theta)\sigma_{-2}(\Theta)^*
\]
with 
\[
\Tr_E(\sigma_{-2}(\Theta)\sigma_{-2}(\Theta)^*)=\sum_{i,j}\sigma_{-2}(\Theta_{ij})\sigma_{-2}(\Theta_{ij})^*
\]
which is identified with $\sum_{i, j}c(\Theta)c(\Theta)^*$ under $\pi$ in (\ref{eq:pi}). See Remark~\ref{rmk:Acd}.
Thus, from the trace formula~(\ref{eq:Connes.trace}),
we have 
\begin{align*}
\Tr^{\Gamma}_{\omega}(\Theta^2)=&\int_{S^*X}\tau\circ\Tr_S\Tr_E\sigma_{-4}(\Theta^2)(x, \xi)dxd\xi\\
=&\sum_{i,j}\int_X\tau\circ\Tr_S(c(\Theta_{ij})c(\Theta_{ij})^*)dx\\
=&\sum_{i,j}\int_X\tau(\|c(\Theta_{ij})\|^2)dx.
\end{align*}
Note that $\mathrm{rank}[H_1(X; \Z)]=4$ is used here to ensure $X$ a $\Gamma=\Z^4$ cover. 
So we obtain the equality:
\[
\Tr^{\Gamma}_{\omega}(\Theta^2)=\sum_{i,j}\int_X \tau(\|c(\Theta_{ij})\|^2). 
\]
\end{proof}
  
  For a $4$-manifold $X$ with a $\Gamma=\Z^4$ cover, we obtain a $\Gamma$-trace formula 
  for  the higher Connes-Yang-Mills  functional:
\begin{theorem}
\label{thm:inf}
 Let $X$ be a closed, oriented, spin, smooth $4$-manifold with $b_1=4$. Let $\Gamma=\Z^4$ be a covering group of $X$. (Existance of a $\Z^4$-cover is ensured by $b_1=4$.)  Let $\Theta_c=\mathbb A_c^2$ 
be the cuarvature of the
classical connection associated to
a quatized connection $\mathbb A$, and put
the  quatized curvature by $\Theta=\mathbb A$.  
Then the formula holds:
\[
\inf_{c(\mathbb A)=\mathbb A_c}\Tr_{\omega}^{\Gamma}(\Theta^2)=\sum_{i,j}\int_{X}\tau(\|(\Theta_{c})_{ij}\|^2).
\]
\end{theorem}

\begin{proof}
In view of Lemmas~\ref{lem} and~\ref{lem:Acdc}, the proof follows essentially from the proof Theorem 14 in~\cite{ConnesYM}. Namely, observe that locally $T_x^*X\otimes T_x^*X\cong \Lambda^2\R^4\oplus S^2\R^4$ is orthogonally decomposed into the symmetric part and the antisymmetric part. By the triangle inequality
with Lemma~\ref{lem:Acdc}, we obtain the estimate:
\[
\tau(\|c(\Theta_{ij})\|^2)\ge\tau(\|A(c(\Theta_{ij}))\|^2)=\tau(\|(\Theta_c)_{ij}\|^2).
\]
Note that all numbers above are nonnegative  because $\Theta^2$ is a positive operator.
The proof then is implied by the previous lemma.
\end{proof}

\begin{theorem}
\label{thm:coincidence}
 Let $X$ be a closed, oriented, spin, smooth $4$-manifold with $b_1=4$. Let $\Gamma=\Z^4$ be a covering group of $X$.  We have $\theta^2\in \cL^{(1, \infty)}(\cH)$ and $\Theta^2\in \cL^{(1, \infty)}(\cE)$ and also the equality holds:
\[
\inf_{c(A)=A_c}\Tr_{\omega}(\theta^2)=\inf_{c(\mathbb A)=\mathbb A_c}\Tr_{\omega}^{\Gamma}(\Theta^2).
\]
Moreover the positive functional is independent of $\omega$.\end{theorem}

\begin{proof}
Note that $\theta^2$ is a pseudo differential operator of order $-4$,
 and $\Theta^2$ is a pseudo differential operator of order $-4$ with coefficient in $C^*_r(\Gamma)$. 
Then $\theta^2\in\cL^{(1, \infty)}(\cH)$ follows from Theorem~1 of~\cite{ConnesYM} and $\Theta^2\in\cL^{(1, \infty)}(\cE)$ follows from Proposition~\ref{prop:pdoDixTr} below.

To show the identity, by Theorem~\ref{thm:inf}, and Theorem~14 in~\cite{ConnesYM}, we only need to  verify the equality:
\[
\int_X\Tr_E\theta_c^2=\int_X\tau\circ\Tr_E(\Theta_c^2). 
\]
Recall that the connection $A_c$ on 
 $E\rightarrow X$
 gives rise to a family of connections $A_{\rho}$  on
 the family of complex vector bundles:
\[
\mathbb{E}=\{\tilde E\otimes_{\Gamma}\C_{\rho}\}_{\rho\in \hat \T^4}
\rightarrow X\times \hat \T^4.
\]
Denote by $\theta_{A_c\otimes\rho}$ the curvature of $A_{\rho}$,
which actually coincides with   $\theta_c=\theta_{A_c}.$
The twisted connection $\mathbb{A}_c$ on $\tilde E\otimes_{\Gamma}C^*_r(\Gamma)\rightarrow X$  is identified with the above family of connections. 
By Fourier transform, evaluation at the  identity in $C^*_r(\Gamma)$ corresponds averaging of elements of $C(\hat \T^4)$ over $\hat \T^4.$ 
So $\tau\circ \Tr_E(\Theta_c^2)$ is to average the family involving curvatures 
$\int_{\hat \T^4}\Tr_E\theta_{A_c\otimes\rho}^2$ on constant $1$ function on $\hat \T^4$, which corresponds to $\delta_e$ in $C^*_r\Gamma.$ 
Now the equality:
\[
\theta_{A_c\otimes\rho}(1)=\theta_{A_c}+(d\rho)(1)=\theta_{A_c}.
\]
 follows since  $\rho$ is a flat connection. 
Set the volume of $\hat\T^4$ to be $1$. We have: 
\[
\tau\circ\Tr_E(\Theta_c^2)=\int_{\hat\T^4}\Tr_E\theta^2_{A_c\otimes\rho}=\Tr_E\theta_{A_c}^2=\Tr_E\theta^2.
\]
and the proof is complete. 
\end{proof}

\begin{definition}
A higher  Connes-Yang-Mills instanton is a connection of the form
$\mathbb{A}_c$ 
which attains local minimum of  the higher Connes-Yang-Mills action functional given by: 
\[
CYM^{\Gamma}({\mathbb A}_c):=16\pi^2\inf_{c(\mathbb A)={\mathbb A}_c}I(\mathbb A).
\] 
\end{definition}

The corollary below follows immediately from Theorem~14 in~\cite{ConnesYM} and Theorem~\ref{thm:coincidence}.

\begin{corollary}
 Let $X$ be a closed, oriented, spin, smooth $4$-manifold with $b_1=4$. Let $\Gamma=\Z^4$ be a covering group of $X$.  Then the  Connes-Yang-Mills functional and its higher analogue coincide:
\[
CYM({A}_c)=CYM^{\Gamma}({\mathbb A}_c).
\]

In particular a Connes-Yang-Mills instanton is equivalent to 
its  higher version.
\end{corollary}
 Lemma \ref{cNT}  gives the following:
 
 \begin{corollary}
 Let $X$ be a closed, oriented, spin, smooth $4$-manifold with $b_1=4$. Let $\Gamma=\Z^4$ be a covering group of $X$. The higher Nahm transform sends the minimizer of the 
higher Connes-Yang-Mills functional  to the minimiser 
of the Connes-Yang-Mills functional.
\end{corollary}

\section{Some properties of Dixmier $\Gamma$-trace}
By definition, Dixmier $\Gamma$-trace is positive and vanishes on the ideal of $\Gamma$-trace class operators. 
We also have the following property for the Dixmier $\Gamma$-trace.
Let $\mathcal{E}$ be a Hilbert $C^*_r(\Gamma)$ module.

\begin{lemma}
For every $A\in \mathcal{L}^{(1,\infty)} (\mathcal{E})$ and $Y\in \cL(\cE)$ bounded, we have:
$$\Tr_{\omega}^{\Gamma} (AY) =\Tr_{\omega}^{\Gamma} (YA).$$
\end{lemma}

\begin{proof}
Because every bounded linear operator on $\cE$ can be written as a linear combination of unitary operators
 (see, for example, Lemma on page 209 VI 1.6 of~\cite{RS} adapted to the case of Hilbert module), we only need to verify  the equality:
\[
\delta_r^{\Gamma}(A)=\delta_r^{\Gamma}(U^*AU) 
\] 
for a unitary operator $U$.
Now 
\[
\delta_r^{\Gamma}(U^*AU)=\sup_P\{\tr(U^*AUP)\}=\sup_P\{\tr(AUPU^*)\}=\delta_r^{\Gamma}(A)
\]
because $U^*PU$ is a projection of rank $r$ if $P$ does. 
\end{proof}

In the following, we will introduce some examples of pseudo differential operators 
arising in the domain of Dixmier $\Gamma$ trace together with an analogue of Connes' trace theorem \cite{ConnesYM} (see also~\cite{AM}) when $\Gamma$ is free abelian. 
We recall some pseudo differential calculus on a closed manifold with coefficient
 in a $C^*_r(\Gamma)$-bundles of finite type in~\cite{BFKM}. 

Let $\cV\rightarrow X$ be a flat $C^*_r(\Gamma)$-bundle over a closed manifold $X$,
 whose fiber is isomorphic to $(C^*_r(\Gamma))^N$ for some $N\in\N$. 
 For example, the tensor 
product of a complex vector bundle $E$ over $X$ with the Mischenko-Fomenko bundle $\gamma=\tilde X\times_{\Gamma}C^*_r(\Gamma)$ forms such a bundle $\cV=E\otimes\gamma$. 

A pseudo differential operator acts on the set of smooth sections: 
\[
\cE^{\infty}:=C^{\infty}(X, \cV)
\] 
and one can take  the  closure to obtain  a Sobolev space $H^{l}(X, \cV)$. 
Denote the Hilbert module over $C^*_r(\Gamma)$:
\[
\cE:=L^2(X, \cV)=H^0(X, \cV).
\]
As in  the classical case, a linear operator:
\[
P: \cE^{\infty}\rightarrow\cE^{\infty}
\]
is a pseudo differential operator of order $m$
 if it can be expressed as:
\[
P=\sum_{j}P_j+R
\]
where $R$ is a smoothing operator and $P_j$ are pseudo differential 
operators with support in the domain of $\psi_i$ 
 for an atlas $\{\psi_i\}$ of  $\cV\rightarrow X$.
We denote by $\Psi^{m}_{\Gamma}(X, \cV)$ the space of pseudo differential operators on  $X$ with coefficient in $\cV$ of order $m$. 

A pseudo differential operator on $X$ with coefficient in the 
$C^*_r(\Gamma)$-bundle $\cV$ can be also  constructed by gluing. 
Let $X=\cup_j\Omega_j$ be a locally finite covering of $X$
 by coordinate neighbourhoods, 
  and $P_j$ be pseudo differential operators of order $m$ on $\Omega_j$. 
Let $\sum_j\psi_j=1$ be partition of unity subordinate to the given covering and let $\phi_j\in C^{\infty}_0(\Omega_j)$ with $\phi_j|_{\supp\psi_j}=1.$ 
Then 
\[
P=\sum_{j}\phi_jP_j\psi_j+R
\] 
is a pseudo differential operator of order $-m$ where 
 $R\in \Psi^{-\infty}_{\Gamma}(X, \cV)$ is a smoothing operator.

One can check  that any pseudo differential operator $P\in\Psi^m_{\Gamma}(X, \cV)$ of order $m$ extends to: 
\[
P: H^m(X, \cV)\rightarrow H^0(X, \cV). 
\]
In particular, any $P\in\Psi^{m}_{\Gamma}(X, \cV)$ with $m\le 0$ extends to a bounded linear operator 
\[
P: \cE\rightarrow\cE.
\]
Moreover $P\in\cK(\cE)$ holds
if  $m\le-\dim X$. In particular, if $m<-\dim X$, $P\in\cK(\cE)$ is a $\Gamma$-trace class operator. 

\begin{example}
Let $X=\T^n$ be the flat torus and $\Delta$ be the Laplacian detemined by the square of the Dirac operator on $\T^n.$ Let 
\[
\Delta_{\gamma}: C^{\infty}(\T^n, S\otimes\gamma)\rightarrow C^{\infty}(\T^n, S\otimes\gamma)
\] 
be the twist of $\Delta$ by the Mischenko-Fomenko line bundle $\gamma$,
induced from $\Delta$. 
Then: 
\begin{equation}
\label{eq:deltaL1infty}
(1+\Delta_{\gamma})^{-\frac{n}{2}}\in\cL^{(1, \infty)}(\cE)
\end{equation}
where $\cE=L^2(\T^n, S\otimes\gamma).$
We will verify this by calculating its Dixmier $\Gamma$ trace explicitly in the next subsection. See Proposition~\ref{prop:eq.torus.Dix}.
\end{example}

\begin{proposition}
\label{prop:pdoDixTr}
Let $\Gamma=\Z^n$. 
Let $X$ be a closed manifold of dimension $n$ with a $\Gamma$-cover.
Let $\cV\rightarrow X$ be a flat 
$C^*_r(\Gamma)$-bundle whose fiber is isomorphic to $(C^*_r(\Gamma))^N$ for some $N\in\N$. 
Let $\cE=L^2(X, \cV)$ be the Hilbert module of $L^2$-sections of $\cV.$
Then every pseudo differential operator $P\in\Psi^{-n}_{\Gamma}(X, \cV)$ of order $-n$ has finite Dixmier $\Gamma$ trace, i.e., 
\[
P\in \cL^{(1, \infty)}(\cE) \quad \text{and} \quad \Tr_{\omega}^{\Gamma}(P)<\infty.
\]
\end{proposition}

\begin{proof}
We will first prove the special case when $X=\T^n$ and $\cV=\gamma$ as in the previous example. 
A pseudo differential operator $P$ of order $-n$ has the form 
\[
P=B(1+\Delta_{\gamma})^{-\frac{n}{2}}
\]
where $B\in\Psi^0_{\Gamma}(\T^n, \cV)$ 
is a bounded linear operator on $\cE$ and 
$(1+\Delta_{\gamma})^{-\frac{n}{2}}\in\Psi^{-n}_{\Gamma}(\T^n, \cV)$ is a compact operator on $\cE$.
From Proposition~\ref{prop:eq.torus.Dix} we have  (\ref{eq:deltaL1infty}). 
Then noting that $\cL^{(1, \infty)}(\cE)$ is an ideal of $\cL(\cE)$, we have 
\[
P\in\cL^{(1, \infty)}(\cE) \qquad \cE=L^2(\T^n, \gamma).
\]
 
Next, in general let $X$, $\cV$ and $P$ be as assumed in the proposition.  Then 
\begin{equation}
\label{eq:decom.pdo}
P=\sum_j\phi_jP\psi_j+R
\end{equation}
where $R$ is a smoothing operator. 
Let us fix a covering $\{\Omega_j\}_j$ of $X$ such that $\cV|_{\Omega_j}$ is trivial.  
To see $\Tr_{\omega}^{\Gamma}(P)<\infty$
in view of~(\ref{eq:decom.pdo}) we only need to verify finiteness:
\[
\Tr_{\omega}^{\Gamma}(\phi_j P\psi_j)<\infty \qquad \forall j.
\]
Embed $\Omega_j$ in $\T^n$ and define the operator to be $0$ on other coordinate patch of $\T^n$. Then $\phi_j P\psi_j$ can be viewed as a pseudo differential operator on $\T^n$ with coefficient in $\cV$. 
But we have already verified: 
\[
\phi_j P\psi_j\in\cL^{(1, \infty)}(\cE)
\]
and hence $P\in \cL^{(1, \infty)}(\cE).$
\end{proof}

The principal symbol $\sigma_P$ of $P\in\Psi^m_{\Gamma}(X, \cV)$ 
can be identified as: 
\[
\sigma_P\in C^{\infty}(S^*X, \End\cV).
\]
Define the $\Gamma$-residue of $P$ by: 
\[
\Res_{\Gamma}(P):=\int_{S^*X}\Tr_{\cV}(\sigma_P(x, \xi))d\nu.
\]
Here $\Tr_{\cV}$ is the composition of matrix trace $\Tr$ 
with the von Neumann trace $\tr$,
 regarding an element of $\End(\cV)$ as a matrix with entries in $C^*_r(\Gamma)$
 locally.

\begin{proposition}
\label{prop:gammaConnesTr}
Under the condition in Proposition~\ref{prop:pdoDixTr}, 
the equality holds:
\[
\Tr_{\omega}^{\Gamma}(P)=\frac{1}{n}\Res_{\Gamma}(P).
\]
\end{proposition}

\begin{proof}
Note that $\Psi^{-n-1}(X, \cV)$ consists of
 $\Gamma$-trace class operators on which $\Tr^{\Gamma}_{\omega}$ vanishes. 
 Therefore $\Tr_{\omega}^{\Gamma}$ is a well-defined linear functional on the quotient space 
$\Psi^{-n}(X, \cV)/\Psi^{-n-1}(X, \cV)$, which is identified with the space of order $-n$ principal symbols $\sigma_P$ in $C^{\infty}(S^*X, \End\cV).$ 
Therefore, $\Tr^{\Gamma}_{\omega}(P)$ depends only on trace of the principal symbol of $P$, 
\[
(x, \xi)\mapsto \Tr_{\cV}[\sigma_P(x, \xi)]
\]
which is a continuous function on $S^*X$.
We denote this function by $f_P$. 
As explained before, 
\[
T: f_P\rightarrow \Tr^{\Gamma}_{\omega}(P)
\]
gives rise to a distribution on $C(S^*X)$.
Note that this distribution is positive.
 By the Riesz-Markov-Kakutani representation theorem on linear functionals on continuous functions, the distribution $T$ is given by a positive measure $\mu$ on $S^*X.$, i.e., $T(f_P)=\int f_P d\mu$ and that is:
\begin{equation}
\label{eq:Connes.trace}
\Tr_{\omega}^{\Gamma}(P)=\int_{S^*X}f_P(x, \xi)d\mu(x, \xi)=\int_{S^*X}\Tr_{\cV}(\sigma_P(x, \xi))d\mu(x, \xi).
\end{equation}

Because $\Tr_{\omega}^{\Gamma}$ is invariant under unitary transformation, the measure $\mu$ is invariant under isometry.
Because both $\Tr_{\omega}^{\Gamma}$ and $\Res_{\Gamma}$ can be reduced locally to an atlas trivialising $\cV$, it is sufficient to show the equality for any closed manifold with trivial $C^*_r(\Gamma)$-line bundle. 
Let $X=S^n.$ 
As explained in~\cite{AM}, the group of isometries $SO(n+1)$ on $\R^{n+1}$ induces a action on $S^*S^n$ as a homogeneous space and the volume form on $S^*S^n$ of the induced Riemannian metric is invariant under the action of $SO(n+1).$
Then uniqueness of invariant measure on $S^*S^n$ shows that the measure $\mu$ is proportional to  the volume form. 
Note that uniqueness of invariant measure on homogeneous spaces can be found in~\cite{N}.
Now the proposition is proved because the scaling constant is completely determined by the example of $\T^n.$ 
\end{proof}

\subsection{$\Z^n$-Dixmier trace for flat torus}

Let $X=(\R/\Z)^n$ be the $n$-torus and 
$D$ be the Dirac operator on $X$. $D$  has point spectrum only.
In particular, associated to
 each element $x$ in the integer lattice $\Z^n$
 is an eigenvector of the Laplacian $\Delta=D^2$ with eigenvalue $\|x\|^2.$
 We call $\Z^n$ the spectral lattice of $D^2$.
The pseudo differential operator $(D^2+1)^{-\frac{n}{2}}$ has  order $-n$ and is a compact operator on $L^2(M, S)$ where $S$ is the spinor bundle.  

 Let $D$ be the Dirac operator on $\T^n$ and $D_{\gamma}$ 
be its twist by the Mischenko-Fomenko line bundle over $\T^n.$ 
 We verify the following:
\begin{proposition}
\label{prop:eq.torus.Dix}
The equality holds:
\[
\Tr_{\omega}((1+D^2)^{-\frac{n}{2}})=\Tr_{\omega}^{\Gamma}((1+D_{\gamma}^2)^{-\frac{n}{2}}).
\]
\end{proposition}
The statement remains true when replacing the operators by any order $-n$ pseudo differential operator obtained from functional calculus of $D, D_{\gamma}.$

 \begin{proof}
 Denote $A=(D^2+1)^{-\frac{n}{2}}.$
In~\cite{ConnesYM}, the Dixmier trace was computed by: 
\[
\Tr_{\omega}(A)=\lim_{V_r\to\infty}\frac{\sum_{\lambda\in\Z^n, \|\lambda\|\le r}(1+\|\lambda\|^2)^{-\frac{n}{2}}}{\log(V_r+1)}=\frac{\Omega_n}{n}=\frac{1}{n(2\pi)^n}\int_{S\T^n}\sigma_{-n}(A)
\]
which is known as Connes' trace formula,
where $V_r=\frac{\Omega_n}{n}r^n$ is the volume of ball of radius $r$ in $\R^n$
and  $\Omega_n$ is the area of the unit sphere $S^{n-1}.$

We recall the contents in Section~\ref{sec:DixmierGamma-trace}.
Consider a sequence of lattice points $x_0, \ldots, x_m, \ldots\in\Z^n$ 
with $\|x_i\|>\|x_j\|$ for all $i>j$, whose union satisfies the coincidence below as sets: 
\[
\{\|x_i\|: i=0, 1, \ldots\}=\{\|x\|: x\in \Z^n\}.
\]   
The operator $A$ has the set of eigenvalues: 
\[
\{(\|x\|^2+1)^{-\frac{n}{2}}: x\in \Z^n\}.
\]
Let $\lambda_1>\lambda_2>\cdots$ be the ordered set of eigenvalues. 
Then: \[
\lambda_i=(\|x_i\|^2+1)^{-\frac{n}{2}}.
\]
The dimension of the eigenspaces admits the formula:
\[
\sum_{j=0}^i\dim E_{\lambda_j}=\#\{x\in\Z^n: \|x\|\le\|x_i\|\}.
\]
Let us denote the right hand side by $N_{0, \|x_i\|}$, 
the number of lattice points in the ball centered at the origin with radius $\|x_i\|.$ 
Denote
the number of lattice points on the sphere centered at the origin with radius $\|x_i\|$
by:
\[
S_{0, \|x_i\|}:=N_{0, \|x_i\|}-N_{0, \|x_{i-1}\|}.
\]

Using lemma  \ref{dim}  we have the equality: 
\[
\mu_t(A)=(\|x_i\|^2+1)^{-\frac{n}{2}}
\]
where $N_{0, \|x_{i-1}\|}\le t < N_{\|x_i\|}$
and $\mu_n$ is the $(n+1)$-th largest eigenvalue counted with multiplicity. 

Therefore, the Dixmier trace of $A$ is calculated by: 
\begin{align*}
\Tr_{w}(A)=&\lim_{t\to\infty}\frac{1}{\log(t+1)}\int_0^t\mu_s(A)ds\\
=&\lim_{N\to\infty}\frac{1}{\log(N+1)}\sum_{n=0}^N\mu_n(A)\\
=&\lim_{i\to\infty}\frac{1}{\log(N_{0,\|x_i\|}+1)}\sum_{j=0}^i
(\|x_j\|^2+1)^{-\frac{n}{2}}S_{0, \|x_j\|} \\
=&\lim_{i\to\infty}\frac{1}{\log N_{0,\|x_i\|}}\sum_{j=0}^i
\|x_j\|^{-n} S_{0, \|x_j\|}\\
=&\frac{\Omega_n}{n} 
\end{align*}
and hence it coincides with 
$\frac{\Omega_n}{n} = \frac{2\pi^{n/2}}{n \Gamma(n/2)}.$

Denote by $\gamma$ the Mischenko-Fomenko line bundle and
 $\tilde S\otimes_{\Z^n}C^*_r(\Z^n)$ the $C^*$-algebra bundle obtained by twisting 
 $S\rightarrow\T^n$ by $\gamma$ over $\T^n$. 
 The Dirac operator  $D$ induces the action $D_{\gamma}$ 
  on the $L^2$-sections: 
\[
\cE=L^2(\T^n, \tilde S\otimes_{\Z^n}C^*_r(\Z^n))
\] 
which is a Hilbert module over $C^*_r(\Z^n).$ 
Under the identification: 
\[
C^*_r(\Z^n)\cong C(\hat \Z^n)=\int^{\oplus}_{\alpha\in\hat\Z^n}\C_{\alpha},
\]
$\cE$ admits a direct integral decomposition: 
\[
\cE=\int^{\oplus}_{\alpha\in\hat\Z^n}L^2(\T^n, \tilde S\otimes_{\Z^n}\C_{\alpha}).
\]
Under this decomposition $D_{\gamma}$ is identified with a family of Dirac operators $D_{\alpha}$ on $H_{\alpha}=L^2(\T^n, \tilde S\otimes_{\Gamma}\C_{\alpha})$ 
parametrised by $\alpha\in\hat\Z^n$:
\[
D_{\gamma}=\{D_{\alpha}\}_{\alpha\in \hat \Z^n}.
\]

$D_{\alpha}$  has point spectrum.
We claim that $D_{\alpha}$ has spectral lattice
which  differs by  shift of $i\alpha$ from that of $D =D_0$,
where $0$ is the trivial representation.
To see this, recall that $\alpha\in\hat \Z^n$ 
gives rise to a representation $\pi_{\alpha}: \Z^n\rightarrow \C^*$ given by: 
\[
\pi_{\alpha}(m)=e^{i\langle\alpha, m\rangle}\qquad m\in\Z^n.
\]
A smooth section of of $\R^n\times_{\Z^n}\C_{\alpha}$ is identified as a function on $\R^n$ with values in $\R^n\times\C_{\alpha}$, i.e., 
$f\in \Gamma(\R^n, \R^n\times\C_{\alpha})^{\Z^n}$ which is  $\Z^n$-invariant 
 in the  sense:
\[
f(x)=f(x+m)e^{i\langle \alpha, m\rangle}\qquad x\in \R^n. 
\]
Then $g(x):=e^{i\langle \alpha, x\rangle}f(x)$ is $\Z^n$-invariant: 
\[
g(x)=g(x+m)\qquad x\in\R^n, m\in \Z^n.
\]
So $g$ is a section of $\R^n\times_{\Z^n}\C_0$ and hence 
$e^{i\langle \alpha, x\rangle}D_{\alpha}f=D g=De^{i\langle\alpha, x\rangle}f.$
This follows from the fact that $D_0$ and $D_{\alpha}$ are viewed as the same operator after identifying these spaces they act on by the unitary operator: 
\[
e^{i\langle\alpha, x\rangle}: L^2(\T^n, \R^n\times_{\Z^n}\C_{\alpha})\rightarrow L^2(\T^n, \R^n\times_{\Z^n}\C_{0}).
\]
Thus, 
\[
D_{\alpha}=e^{-i\langle\alpha, x\rangle}De^{i\langle\alpha, x\rangle}=D+i\alpha.
\]
The claim is proved.

Let $A=(1+D_{\gamma}^2)^{-\frac{n}{2}}
\in\cK(\cE)$ be the positive compact endomorphism
 on the  
Hilbert $C^*_r(\Z)$-module. It is represented by: 
\[
h=\{P_{\alpha}\}_{\alpha\in\hat\Z^n} 
\]
where $ P_{\alpha}=(D_{\alpha}^2+1)^{-\frac{n}{2}}$
is a family of compact operators. 
Denote by $E_n$ the $1$-dimensional eigenspace of $D$ associated to $n\in\Z^n.$
Then $P_{\alpha}$ admits spectral decomposition: 
\[
P_{\alpha}=\oplus_{n\in\Z^n}(||n+i\alpha||^2+1)^{-\frac{n}{2}}E_n.
\]
For $s\ge0$, we have:
\[
\chi_{(s,\infty)}(A)=\int_{\alpha\in\hat\Z^n}\chi_{(s, \infty)}(P_{\alpha})d\alpha. 
\]
Note that the integrant $\chi_{(s, \infty)}(P_{\alpha})$ is a finite rank projection, whose operator trace is equal to: 
\[
\Tr(\chi_{(s, \infty)}(P_{\alpha}))=\#\{x\in\Z^n:(||x+i\alpha||^2+1)^{-\frac{n}{2}}>s\}
\]

The von Neumann trace $\tr$ on the dual under Fourier transform is to integrate over $\alpha\in\hat\Z^n:$
\[
\tr(\chi_{(s, \infty)}(A))=\int_{\alpha\in\hat\Z^n}\Tr(\chi_{(s, \infty)}(P_{\alpha}))d\alpha.
\]

By definition: 
\begin{align*}
\mu_t(A)=&\inf\left\{s\ge0: \int_{\alpha\in\hat\Z^n}\#\{m\in\Z^n:(||m+i\alpha||^2+1)^{-\frac{n}{2}}>s\}d\alpha \leq t\right\}\\
=&\inf\left\{(\|x\|^2+1)^{-\frac{n}{2}}: \int_{\alpha\in\hat\Z^n}\#\{m\in\Z^n: \|m+i\alpha\|\le\|x\|\}d\alpha\leq t\right\}.
\end{align*}

Denote the integrant 
$\#\{m\in\Z^n: \|m+i\alpha\|\le\|x\|\}$ by $N_{\alpha, \|x\|}$.
Observe that: 
\[
\int_{\alpha\in\hat\Z^n}N_{\alpha, \|x\|}d\alpha\sim \Vol B_{\|x\|}
\]
as $||x|| \to \infty$,
where $B_{\|x\|}$ is volume of the ball of radius $\|x\|$ in $\R^n.$

Thus we obtain: 
\[
\mu_t(A)=\inf\left\{(\|x\|^2+1)^{-\frac{n}{2}}: \Vol B_{\|x\|}\leq t\right\}.
\]
Because $\Vol B_{\|x\|}=\frac{\Omega_n}{n}\|x\|^n$, and 
$(\|x\|^2+1)^{-\frac{n}{2}}\sim \|x\|^{-n}$
when $x$ is large, we have  
\[
\mu_t(A)\sim \frac{1}{t}\frac{\Omega_n}{n}
\]
when $t$ is large.
The Dixmier $\Gamma$-trace is then equal to: 
\[
\Tr_{\omega}^{\Gamma}(A)=\lim_{t\to\infty}\frac{1}{\log(t+1)}\int_1^t\mu_s(A)=\frac{\Omega_n}{n}\lim_{t\to\infty}\frac{\log(t)}{\log(t+1)}=\frac{\Omega_n}{n}.
\]
\end{proof}


\end{document}